\documentclass{article}

\usepackage{arxiv}

\usepackage[utf8]{inputenc} 
\usepackage[T1]{fontenc}    
\usepackage{hyperref}       
\usepackage{url}            
\usepackage{booktabs}       
\usepackage{amsfonts}       
\usepackage{nicefrac}       
\usepackage{microtype}      
\usepackage{lipsum}
\usepackage{amsmath,amsthm,amssymb}
\usepackage[all]{xy}
\usepackage{graphicx}
\usepackage{mathrsfs}
\usepackage{ upgreek}
\theoremstyle{plain}
\theoremstyle{remark}

\newtheorem*{convention*}{Convention}
\theoremstyle{plain}
\newtheorem{theorem}{Theorem}[section]

\newtheorem{corollary}[theorem]{Corollary}
\theoremstyle{definition}
\newtheorem{definition}[theorem]{Definition}
\newtheorem{example}[theorem]{Example}

\newcommand{\bbZ}{{\mathbb Z}}
\newcommand{\bbR}{{\mathbb R}}

\DeclareMathOperator{\Inv}{Inv}

\title{Conley index methods detecting bifurcations in a modified van der Pol oscillator appearing in heart action models}

\author{Ibrahim Jawarneh\\
  Department of Mathematics\\
 Al-Hussein Bin Talal University\\
Ma'an, P.O. Box (20), 71111, Jordan\\
  \texttt{ibrahim.a.jawarneh@ahu.edu.jo} \\
   \And
  Ross Staffeldt \\
  Department of Mathematical Sciences\\
 New Mexico State University\\
 Las Cruces, NM 88003, USA\\
  \texttt{ross@nmsu.edu} \\
}

\begin{document}
\maketitle

\begin{abstract}
In this paper, a modified van der Pol oscillator equation is considered which appears in several heart action models. We study its global dynamics and verify many interesting bifurcations such as a Hopf bifurcation, a heteroclinic saddle connection, and a homoclinic saddle connection. Some of these bifurcations are detected by using Conley index methods. We demonstrate how the study of connection matrices and transition matrices shows how to select interesting parameter values for simulations.
\end{abstract}


\section{Introduction}\label{Introduction}
The conducting system of the human heart contains a sino-atrial node (SA), an atrio-ventricular node (AV), and the His-Purkinje system \cite{Johnson1997}. Since these elements exhibit oscillatory behaviour, they can be modeled as  nonlinear oscillators \cite{Bub1995}. This approach is not suitable for the investigation of the cardiac conducting system at a cellular level, but it allows a global analysis of heartbeat dynamics by investigating interactions between the elements of the system. 

Several models of the cardiac pacemaker are based on the relaxation van der Pol oscillator \cite{van1928}, given by
\begin{equation} \label{BasicvdPheart}
\frac{d^2x}{dt^2} + \alpha(x^2 - 1) \frac{dx}{dt} + \omega^2 x = 0,
\end{equation}
where $\alpha$ is the damping constant and $\omega$ is the frequency. Since the work of van der Pol and van der Mark \cite{van1928}, other models have been proposed based on relaxation oscillators \cite{West1985,Katholi1977,Bernardo1998}, and mainly focused on the interaction between AV and SA nodes. Although they provide many interesting results, these models are not able to reproduce some important features of real cardiac action potentials. For this reason a modified van der Pol  oscillator is proposed in \cite{Postnov1999} which focused on the synchronization of oscillators. In this oscillator the harmonic force $\omega^2 x$ has been replaced by a cubic Duffing term, and is given by 
\begin{equation} \label{mvpsystem,d,2d}
\frac{d^{2}x}{dt^{2}} + \alpha(x^{2}- \theta) \frac{dx}{dt}+ \frac{x(x+d)(x+2d)}{d^{2}}=0, \quad  d,\theta, \alpha > 0, \quad \theta < d.
\end{equation}
This equation has an unstable focus at $x = 0$, a saddle at $x = -d$, and a stable node at $x = -2d$. Although the system undergoes a global homoclinic bifurcation for values of $\theta$  high enough, the distance between the saddle and the node cannot be changed.

In \cite{Krzysztof2004} a model has been developed which is more physiologically accurate, and thus can be used as a basis for further investigation of the dynamics of the heartbeat. This model is based on a modified van der Pol  oscillator \eqref{mvpsystem,d,2d} by introducing a new parameter $e$ as follows:
\begin{equation} \label{mvpsystem,d,e}
\frac{d^{2}x}{dt^{2}} + \alpha(x^{2}- \theta) \frac{dx}{dt}+ \frac{x(x+d)(x+e)}{de}=0, \quad d,e,\theta, \alpha>0, \quad \theta < d.
\end{equation}
This modification does not change the qualitative structure of the phase space of the system. 
The change of parameter $e$ results in a change of the depolarization period. When the nodal point approaches the saddle, then the trajectory tends to spend more time in the vicinity of the saddle, and vice versa.  When the node is moved away from the saddle, the intervals between action potentials become shorter. Varying the parameter $\alpha$ slightly impacts the period of oscillations by a small amount, and varying $\alpha$ over a wide range also changes the shape of the oscillations.  We take $\alpha = 1$ to make our work with the equilibria easier.

We extend the range of the parameters $d,e$ and $\theta$ to get interesting qualitative results, e.g., Hopf bifurcations, and to use Conley index methods to detect heteroclinic saddle connections and homoclinic saddle connections.

In this paper we consider system (\ref{mvpsystem,d,e})  with $\alpha = 1$, $ \theta \in \bbR$, $d,e \in \bbR \setminus \{0\}$   where $\left|d \right| \leq \left| e \right| $, and $d \neq e$ as we aim to get three distinct equilibrium points. Then our new  modified van der Pol oscillator equation is
\begin{equation} \label{mvpsystem extended}
\frac{d^{2}x}{dt^{2}} + (x^{2}- \theta) \frac{dx}{dt}+ \frac{x(x+d)(x+e)}{de} = 0.
\end{equation}

Equation \eqref{mvpsystem extended} can be written  as a first order system
\begin{equation} \label{mvpsystem extended 1st order}
   \begin{aligned} 
    \frac{dx}{dt}& = y,   
    \\
    \frac{dy}{dt}& = - (x^{2}- \theta) y - \frac{x(x+d)(x+e)}{de}.
\end{aligned} 
\end{equation}

This paper is organized as follows$\colon$ In the next section, we study the stability of the equilibria of system \eqref{mvpsystem extended 1st order}.  In section \ref{The Conley index and Connection Matrices} we recall the Coley index theory, connection matrices, transition matrices and some related corollaries and theorems. Section \ref{Bifurcations and Conley index} deals with bifurcations of the model through the cases of the condition $\left|d \right| \leq \left| e \right| $ with $d \neq e$. Numerical examples of the model are provided explaining prospective bifurcations like heteroclinic saddle connections and homoclinic saddle connections. Then methods using the Conley index are applied to detect the bifurcations. 

\section{Steady states and  their stability}\label{steady state}
To study the stability of the equilibria of system \eqref{mvpsystem extended 1st order}, note that its Jacobian matrix takes the form
\begin{equation} \label{Jacobian Heart}
J =
  \begin{bmatrix}
  0 & 1 \\
-2xy - \frac{1}{de}\Big((x+d)(x+e) + x(x+e) + x(x+d) \Big) &\bigl(\theta -x^2 \bigr)
\end{bmatrix}  .
\end{equation}
The equilibria of system \eqref{mvpsystem extended 1st order}  are  solutions of the  equations
 \begin{align} 
    0  & = y,  \label{mvpsystem extended 1st order 1st eqn}
    \\
    0  & = - (x^{2}- \theta) y - \frac{x(x+d)(x+e)}{de}, \label{mvpsystem extended 1st order 2nd eqn}
\end{align}  
and are given by $E_0 =(0,0)$, $E_1 = (-d,0)$, and $E_2 = (-e,0)$.

We first consider the stability $E_0$.  We will require the following theorem, which gives a Lyapunov test for a weak sink or source in  a Hopf bifurcation and  determines the stability of the system at particular values of $\theta$.
\begin{theorem}[Theorem 9.2.3, \cite{Hubbard1995}] \label{Liapunovtest for weak sink or source in Hopf} 
 If $X' = f_{\alpha} (X)$ is a one-parameter family of differential equations in the plane such that for $\alpha = \alpha_0$ the equations can be written 
 \begin{align*}
      x'& =  y + \mu_{2,0}x^2 + \mu_{1,1}xy + \mu_{0,2}y^2 + \mu_{3,0}x^3 + \mu_{2,1}x^2y+\mu_{1,2}xy^2+\\
      &\mu_{0,3}y^3 + \cdots,
      \\
      y'& =   -x + \nu_{2,0}x^2 + \nu_{1,1}xy + \nu_{0,2}y^2 + \nu_{3,0}x^3 + \nu_{2,1}x^2y+\nu_{1,2}xy^2+\\
      &\nu_{0,3}y^3 + \cdots,
 \end{align*}
 define the Lyapunov coefficient
 \begin{align}
    L &\equiv  3\mu_{3,0} + \mu_{1,2} + \nu_{2,1} +3\nu_{0,3} - \mu_{2,0}\mu_{1,1} +\nu_{1,1}\nu_{0,2} \notag
    \\
    &-2\mu_{0,2}\nu_{0,2} -\mu_{0,2}\mu_{1,1} +2\mu_{2,0}\nu_{2,0} + \nu_{1,1}\nu_{2,0}.
 \end{align} \label{Liapunov coefficient}
 The sign of $L$ determines behavior of solutions near the equilibrium at the origin.
 \begin{enumerate}
     \item If $L$ is positive, the origin is a weak source for $\alpha = \alpha_0$, and if $L$ is negative, the origin is a weak sink. If $L$ is zero, you cannot tell.
     \item If $L > 0$, the differential equation will have an unstable limit cycle for all values of $\alpha$ near $\alpha_0$ for which the origin is a sink, and if $L < 0$, it will have a stable limit cycle for all values of $\alpha$ near $\alpha_0$ for which the origin is a source. \qed
 \end{enumerate}
\end{theorem}
The Jacobian matrix evaluated at the equilibrium point $E_0 = (0,0)$ is
\begin{equation} \label{JacobianE_0 Heart}
J(E_0) =
  \begin{bmatrix}
  0 &  1\\
-1 & \theta
\end{bmatrix}  .
\end{equation}
We summarize the stability of $E_0$ in the following theorem.
\begin{theorem} \label{stability E_0 Heart}
 The equilibrium point $E_0 = (0,0)$ is unstable if $\theta > 0$, and stable if $\theta \leq 0$. Moreover, there is a stable limit cycle  for all values of $\theta$ near $\theta = 0$ for which the equilibrium point $E_0 = (0,0)$ is a source.
\end{theorem}
\begin{proof}
From the Jacobian matrix \eqref{JacobianE_0 Heart}, the eigenvalues of $J(E_0)$ are
\begin{equation} \label{eigenvalues of E_0 Heart}
 \lambda_{1,2} = \frac{\theta \pm \sqrt{\theta^2 - 4}}{2}. 
\end{equation}
Notice that the eigenvalues of the equilibrium point $E_0 = (0,0)$ depend only on the parameter $\theta$. When $ \theta < 0$, the real parts of the eigenvalues are negative, so $E_0$ is stable. Similarly, when  $\theta  > 0$, the real parts of the eigenvalues are positive, and $E_0$ is unstable. 

For $\theta = 0$, we apply theorem \ref{Liapunovtest for weak sink or source in Hopf}.  System \eqref{mvpsystem extended 1st order} can be written as
\begin{equation} \label{our system at theta = 0}
    \begin{aligned} 
    \frac{dx}{dt}& = y,   
    \\
    \frac{dy}{dt}& = -x - \frac{(d+e)x^2}{de} - \frac{x^3}{de} - x^2 y.
\end{aligned} 
\end{equation}
and we find the Lyapunov coefficient provided by theorem \ref{Liapunovtest for weak sink or source in Hopf} is $ L = -1 < 0$. This tells us that $E_0$ is a weak sink for $\theta = 0$, so it is stable.  Moreover, the system \eqref{mvpsystem extended 1st order} will have a stable limit cycle for all values of $\theta$ near $\theta = 0$ for which $E_0$ is a source.
\end{proof}

We summarize the stability of $E_{1}=(-d,0)$ in the following theorem.
\begin{theorem}  \label{stability E_1 Heart}
If  $\theta - d^2 \neq 0$, then the equilibrium point $E_1 = (-d,0)$ is  a saddle.
\end{theorem}
\begin{proof}
The Jacobian matrix evaluated at $E_1 = (-d,0)$ is
\begin{equation} \label{JacobianE_1 Heart}
J(E_1) =
  \begin{bmatrix}
  0 &  1\\
(e-d)/e & \theta -d^2
\end{bmatrix},
\end{equation}
and so the eigenvalues of $J(E_1)$ are
\begin{equation} \label{eigenvalues of E_1 Heart}
 \lambda_{1,2} = \frac{(\theta -d^2)\pm 
 \sqrt{(\theta -d^2)^2 + 4\bigl(1 - (d/e)\bigr)}}{2}. 
\end{equation}
Since in system \eqref{mvpsystem extended 1st order} we considered $\left|d \right| \leq \left| e \right|$ and $d \neq e$, we get $ - \left|d \right|/  \left| e \right| \geq -1$, so $(1 - d/e) \geq (1 - \left|d \right|/  \left| e \right|) \geq 0 $. Then  the term $4\bigl(1 - (d/e)\bigr) > 0$, which implies 
\begin{equation*}
  \sqrt{(\theta -d^2)^2 + 4\bigl(1 - (d/e)\bigr)} > \left| \theta -d^2 \right| \geq 0.  
\end{equation*}
 So the smaller eigenvalue
 \begin{equation}
\lambda_{1} = \frac{(\theta -d^2) - \sqrt{(\theta -d^2)^2 + 4\bigl(1 - (d/e)\bigr)}}{2} < \frac{(\theta -d^2) - \left| \theta -d^2 \right| }{2} \leq 0.    
\end{equation}
Since the larger eigenvalue is always positive, the equilibrium point $E_1 = (-d,0)$ is a saddle.
\end{proof}
Finally, We summarize the stability of $E_{2}$ in the following theorem.
\begin{theorem} \label{stability E_2 Heart}
If  $\theta - e^2 \neq 0$, then the equilibrium point $E_2 =(-e,0)$ is
\begin{enumerate}
    \item a saddle,  if $d$ and $e$ have different signs.
    \item  stable,  if $d$ and $e$ have the same sign with $\theta - e^2 < 0$.
    \item unstable,  if $d$ and $e$ have the same sign with $\theta - e^2 > 0$.
\end{enumerate}
\end{theorem}
\begin{proof}
The Jacobian matrix evaluated at $E_2 = (-e,0)$ is
\begin{equation} \label{JacobianE_2 Heart}
J(E_2) =
  \begin{bmatrix}
  0 &  1\\
(d-e)/d & \theta -e^2
\end{bmatrix}.
\end{equation}
Thus the eigenvalues of $J(E_2)$ are
\begin{equation} \label{eigenvalues of E_2 Heart}
 \lambda_{1,2} = \frac{(\theta -e^2)\pm \sqrt{(\theta - e^2)^2 + 4\bigl(1 - (e/d)\bigr)}}{2}. 
\end{equation}
As $\left|d \right| \leq \left| e \right|$ and $d \neq e$, 
we have three cases:

First, if  $d$ and $e$ have different signs, $4\bigl(1 - (e/d)\bigr) > 0$. Then the term 
\begin{equation*}
    \sqrt{(\theta -e^2)^2 + 4\bigl(1 - (e/d)\bigr)} > \left| \theta -e^2 \right| > 0.
\end{equation*}
The smaller eigenvalue
\begin{equation}
\lambda_{1} = \frac{(\theta - e^2) - \sqrt{(\theta - e^2)^2 + 4\bigl(1 - (e/d)\bigr)}}{2} < \frac{(\theta -e^2) - \left| \theta - e^2 \right| }{2} \leq 0.  
\end{equation}
Since  larger eigenvalue is always positive, the equilibrium point $E_2 = (-e,0)$ is a saddle.

When $d$ and $e$ have the same sign the term  $4\bigl(1 - (e/d)\bigr) < 0$ so the other two cases are happened as follows$\colon$

For  $\theta - e^2 < 0$, the real parts of the eigenvalues are negative. Thus, the equilibrium point $E_2 = (-e,0)$ is stable. If $\theta - e^2 > 0$ then the real parts of the eigenvalues are positive. Therefore, the equilibrium point $E_2 = (-e,0)$ is unstable.
\end{proof}

\section{Conley index methods}
\label{The Conley index and Connection Matrices}
In this section we state  basic definitions and theorems that are related to the connection matrices and methods of Conley index to detect some interesting bifurcations like heteroclinic saddle connection and homoclinic saddle connection. For a more complete discussion of  connection matrices  and  the Conley index the reader is referred to \cite{Conley78,Salamon85,Franzosa1,Franzosa2,Franzosa3,Reineck88,Mischaikow88,Jawarneh18}.

Throughout this paper,  $S$ will always denote an isolated invariant set and  the homological Conley index of $S$, $CH_q(S)$, is the relative homology $H_q(N,L)$, where $q$ is a homology index, $(N,L)$ is  an index pair of $S$ where $L \subset N \subset \mathbb R^n$, such that
\begin{enumerate}
\item $S = \Inv(\overline{N\setminus L})$, and $\overline{N\setminus L}$ is an isolating neighborhood  of $S$ in  $X$;
\item $L$ is positively invariant in $N$; that is, given $x \in L$ and $\varphi ( x,[0, t]) \subset N$, then $\varphi ( x,[0, t]) \subset L$;
\item  $L$ is an exit set for $N$; that is given $x \in N$ and $t_{1} > 0$ such that     $\varphi(x,t_{1}) \not \in N$, then there exists $t_{0} \in  [0, t_{1}]$ such that $\varphi (x,[0, t_0] ) \subset N$ and $\varphi (x,t_{0}) \in L$.
\end{enumerate}

The homotopical Conley index of $S$, $h(S)$, is defined as the pointed homotopy type of the quotient space $N/L$. Some necessary features of the  Conley index are listed in the following theorems.
\begin{theorem}[Theorem 3.13, \cite{Mischaikow02}] 
\label{homology of unstable}
Let $p$ be a  hyperbolic equilibrium point with an unstable manifold of dimension $n$. Then
\begin{equation*}
    CH_i(p) \cong \Bigg\{ \begin{tabular}{cc} $\bbZ_2$, & \text{if $i = n$}, \\
  $0$, &  \text{otherwise. \qed}
  \end{tabular}   
\end{equation*}
\end{theorem}

The following result provides a necessary condition for asymptotic stability of an equilibrium point by using the Conley index.

 \begin{theorem}[Theorem 9,  \cite{Moulay11}]  \label{Conley Index Condition for Equilibria}
(Conley Index Condition for Equilibria). Let  p be an asymptotically stable then
\begin{equation*}
    CH_i(p) \cong
\Bigg\{ 
\begin{tabular}{cc}
$\bbZ_2$, & \text{if $i = 0$,} 
\\
  $0$, &  \text{otherwise. \qed}
  \end{tabular}
\end{equation*}
 \end{theorem}
 
If there exists $i\geq$ 1 such that $CH_i(p) \cong \bbZ_2$ then $p$ cannot be asymptotically stable. The following corollary can be used to compute the Conley index of an hyperbolic equilibrium.

\begin{corollary} [Corollary 10, \cite{Moulay11}]
Let $p$ be a hyperbolic equilibrium point and let $\{\lambda _j \}_{1\leq j _\leq n}$ be the eigenvalues of the Jacobian matrix at $p$ such that
$Re(\lambda_q) > 0$ for all  $1\leq q \leq k \leq n$
then
\begin{equation*}
 CH_i(p) \cong \Bigg\{
  \begin{tabular}{cc}
$\bbZ_2$, & \text{if $i = k$}, \\
  $0$,   &  \text{otherwise}. \qed
  \end{tabular}   
\end{equation*}
\end{corollary} 

The Conley index of the periodic orbit is given in the following corollaries.
\begin{corollary}[Corollary 3.17, \cite{Mischaikow02}]
Let $S$ be a hyperbolic invariant set that is diffeomorphic to a circle. Assume that      $S$ has an oriented unstable
manifold of dimension $n+1$. Then 
\begin{equation*}
CH_i(S) \cong \Bigg\{
  \begin{tabular}{cc}
$\bbZ_2$, & \text{if $i= n, n+1$}, \\
  $0$, & \text{otherwise}. \qed
 \end{tabular}    
\end{equation*}
\end{corollary}
\begin{corollary}
If $\mathcal{O}$ is an asymptotically stable periodic orbit then 
\begin{equation*}
    CH_i (\mathcal{O}) \cong \Bigg\{
  \begin{tabular}{cc}
$\bbZ_2$, & \text{if $i = 0, 1 $}, \\
  $0$, &  \text{otherwise. \qed}
  \end{tabular} 
\end{equation*}
\end{corollary}

If $S_1$ and $S_2$ are isolated invariant sets, then we define the set of connecting orbits from $S_2$ to $S_1$ as
\begin{equation} \label{connectingorbits}
C(S_2, S_1; S) := \{ x \in  S \mid \omega(x)  \subset  S_1 \; \text{and} \; \alpha (x)  \subset  S_2 \},
\end{equation}
where $\omega(x)$ and $\alpha (x)$ denote the omega and alpha limit sets of $x$, respectively. 

A \emph{partial order} on a set $\mathcal{P}$ is a relation $<$ on $\mathcal{P}$ that satisfies:
\begin{enumerate}
     \item  $\pi < \pi$ never  holds for $\pi \in \mathcal{P}$;
     \item if $\pi <  \pi^{'}$  and  $\pi^{'} < \pi^{''}$ then  $\pi  < \pi^{''}$.
\end{enumerate}
 A subset $I \subset \mathcal{P}$ is called an \emph{interval} if  $\pi < \pi^{''} < \pi^{'}$ with $\pi$, $\pi^{'} \in I$ implies $\pi^{''} \in I$. The set of intervals in $<$ is denoted $I(<)$.
 
An interval  $I$ $\in I(<)$ is called an \emph{attracting interval} if $ \pi \in  I$ and $\pi^{'} < \pi $ imply $\pi^{'}\in I $.  The set of attracting intervals in $<$ is denoted $A(<)$.
  
Points $\pi$, $\pi^{'} \in \mathcal{P}$ are called \emph{adjacent} if $\{\pi,\pi^{'} \} \in I(<)$, i.e., $\{\pi,\pi^{'} \}$ or $\{\pi^{'}, \pi \}$ is an interval in $\mathcal{P}$.

\begin{definition} \label{Morse set decomposition}
Let $S$ be a compact invariant set (not necessarily isolated). A ($<$-ordered) \emph{Morse decomposition} of $S$ is a collection $M(S) = \{ M(p) \mid p \in \mathcal{P} \}$ of mutually disjoint compact invariant subsets of $S$ such that if  $x \in S \setminus \underset{p \in \mathcal{P}} {\bigcup} M(p)$, then there exist $q < p \quad (p, q \in \mathcal{P})$ with $\omega(x) \subset M(q)$ and $\alpha(x) \subset M(p)$, i.e., $ x \in C\bigl(M(p),M(q)\bigr)$.
\end{definition}

In general, any ordering on $\mathcal{P}$ satisfying the above property is called admissible (for the flow). The invariant sets, $M(p)$, are called Morse sets.  Moreover, if $S$ is isolated, then each $M(p)$ is also isolated.

For each $I \in I(<)$, let
\begin{equation*}
    M(I)= \Bigl( \bigcup_{\pi \in I} M(\pi) \Bigr)  \cup \Bigl( \bigcup_{\pi, \pi^{'} \in I} C\bigl(M(\pi^{'}),M(\pi)\bigr) \Bigr).
\end{equation*}
We call $M(I)$ a Morse set of the admissible ordering $<$ of $M$. The collection $ \{ M(I)\mid I \in I(<) \}$ of Morse sets of the admissible ordering $<$ is denoted by $MS(<)$. If I $\in A(< )$, then $M(I)$ is an attractor in $S$, and $M(\mathcal{P}\setminus I)$ is its dual repeller. Since $M(I)$ is isolated invariant set, $h\bigl(M(I)\bigr)$ is defined. Let $CH_q(I) = CH_q\bigl(h\bigl(M(I)\bigr);\bbZ_2 \bigr)$ be the singular homology of the pointed space $h\bigl(M(I)\bigr)$. In particular, given $p \in \mathcal{P}$, $CH_q(p) = CH_q\bigl(h\bigl(M(p)\bigr);\bbZ_2 \bigr)$.
If the collection of invariant sets $\{M (p)\mid p \in (\mathcal{P}, <)\}$ is a Morse decomposition of $S$ with admissible ordering $<$, then an associated connection matrix is a linear map
\begin{equation*}
    \Delta: \underset{p\in \mathcal{P} } {\bigoplus} CH_*(M(p)) \rightarrow \underset{p\in \mathcal{P} } {\bigoplus} CH_*(M(p)),
\end{equation*}
where
$\Delta(p,q): CH_*(M(q)) \rightarrow CH_*(M(p))$
is the corresponding $(p,q)$-component of $\Delta$ such that the following conditions are satisfied:
\begin{enumerate}
\item  $\Delta$ is strictly upper triangular,  so that if $\Delta(p,q) \neq 0$ implies $p < q$.
\item $\Delta$ is a boundary map if each
$\Delta(p,q)$ is of degree $-1$ \\
and $\Delta^2 = 0$.
\item For every interval $I$,
$ CH_*(M(I))  \cong \frac{Ker \Delta(I)}{Im \Delta(I)}$.
\end{enumerate}
The most important properties of the connection matrices are stated in the following theorems.
\begin{theorem}[Proposition 5.3, \cite{Franzosa2}]
 If $\{\pi,\pi^{'} \} \in$ $I(<)$ and $\Delta(\pi, \pi^{'}) \neq 0$, then $C\bigl(M(\pi^{'}), M(\pi)\bigr) \neq \emptyset$. \qed
\end{theorem}
The following theorem  helps us to determine if the boundary map between the Conley invariants of two Morse sets is isomorphism or zero.
\begin{theorem}[Compare  theorem  3.4, \cite{McCord97}]\label{even odd orbits mod 2} Let $M(S) = \{ M(1), M(0) \mid 1 > 0 \}$ be a Morse decomposition consisting of hyperbolic critical points. Let $C(M(1), M(0))$ consist of exactly $k$ heteroclinic orbits which arise as the transverse intersection of the stable and unstable manifolds of $M(0)$ and $M(1)$, respectively. Then $\Delta (1,0) = k  \mod (2) $.\qed
\end{theorem}

As we are interested in studying a parametrized family of differential equations, we need to know how connection matrices at different parameter values are related. This is the purpose of the following theorem.
\begin{theorem}[\cite{Mischaikow89}]
Let $S(\theta)$ be isolated invariant sets related by continuation for all parameter values $\theta$. If $\Delta(\theta)$ and $\Delta(\theta')$ are connection matrices for Morse decompositions of $S(\theta)$ and $S(\theta')$ respectively, then there exists a matrix, $T$, consisting of degree $0$ maps such that 
\begin{equation} \label{Delta^2 = 0}
\Delta(\theta) T +T \Delta(\theta') = 0.
\end{equation}
\end{theorem}
The matrix $T$ is called a transition matrix and its entries give information about the existence of connecting orbits for the parameter values between $\theta$ and $\theta'$.

Because we are interested in heteroclinic and homoclinic saddle connections, the basic results of how we can detect them may be summarized as follows.
\begin{itemize}
    \item Heteroclinic Connections$\colon$ If the transition matrix entry 
$T(M^0_{\pi},M^1_{\rho})$ is a non-zero isomorphism,
then there exists $\theta^* \in ( \theta,\theta')$,
such that $C(M^{\theta^*}_{\rho}, M^{\theta^*}_{\pi}) \neq \emptyset$, where $M^0_{\pi}$ is the Morse set at $\theta$ with a partial order $\pi \in \mathcal{P}$ and $M^1_{\pi}$ is the Morse set at $\theta'$ with a partial order $\rho \in \mathcal{P}'$.  That is, there is a connection from  $M_\rho^{\theta^*}$ to $M_\pi^{\theta^*}$, see corollary 5.8, \cite{Jawarneh18} and theorem 3.13,\cite{Reineck88}.
\item Generalized Homoclinic Connections$\colon$ If there exists homology dimension $q$ such that $T_{q}(M_\pi^0,M_\pi^1)$, is not an isomorphism, then $M_\pi^{\theta^*}$ is not a Morse set for  some $\theta^* \in ( \theta,\theta') $ and, when this occurs, 
 then there exists a generalized homoclinic orbit to $M_\pi^{\theta^*}$, see corollary 4.4,\cite{Mischaikow88}.
\end{itemize}

\section{Heteroclinic and homoclinic bifurcations and the Conley index}\label{Bifurcations and Conley index}
\hspace{\parindent}
Many  bifurcations of the system \eqref{mvpsystem extended 1st order}  will be verified in this section by Conley index methods. We use Maple to reveal these bifurcations by fixing the parameters $d$ and $e$ for different cases depending on the signs of $d$ and $e$, then playing with the value of the parameter $\theta$ over particular intervals in $\bbR$. When we see a kind of bifurcation over a specific interval, we modify this interval to be smaller and smaller, and play with the parameter $\theta$ again over the new smaller interval to get $\theta_b$, the value of $\theta$ before bifurcation, and $\theta_a$, the value of $\theta$ after bifurcation. Then we define $\theta_*$ which is the value of $\theta$ at bifurcation.
To construct connection matrices, we need the basic information stated in table \ref{table basic information}.
\begin{table}
    \centering
  \begin{tabular}{ |c|c|c| } 
 \hline
  If $p$ is a sink  & $CH_q(p) = \bbZ_2$ if $q = 0$.  \\
 \hline
  If $p$ is a saddle  in the  plane & $CH_q(p) = \bbZ_2$ if $q = 1$  \\ 
 \hline
 If $p$ is a source in the plane & $CH_q(p) = \bbZ_2$ if $q = 2$ \\ 
 \hline
  If $\mathcal{O}$ is a stable limit cycle & $CH_q(\mathcal{O}) = \bbZ_2$ if $q = 0$ or $1$\\
  \hline
\end{tabular}
    \caption{Conley index of Morse sets related to the system \ref{mvpsystem extended 1st order}}
    \label{table basic information}
\end{table}

We notice that when the parameters $d$ and $e$ are the same sign, there is a prospective homoclinic saddle connection at the equilibrium point $E_1 = (- d,0)$, which results in the vanishing of the limit cycle around the origin. For example, consider system \eqref{mvpsystem extended 1st order} with $d = 0.5$ and $e = 2$.
\begin{example} \label{example of homo. d and e same signs}
\begin{equation} \label{Homo. d = 0.5, e = 2}
   \begin{aligned} 
    \frac{dx}{dt}& = y,   
    \\
    \frac{dy}{dt}& = - (x^2- \theta) y - x(x+0.5)(x+2).
\end{aligned} 
\end{equation}
\end{example}
Using Maple, we find that when $\theta$ moves from $ \theta_b = 0.02$ to $\theta_a = 0.04$, we can expect a homoclinic saddle connection at the equilibrium point $E_1 = (- 0.5,0)$ that goes around the origin and back to $E_1$, as  illustrated in figure \ref{figure.Homo. example1}. The homoclinic bifurcation happens at $\theta = \theta_* \in (0.02 , 0.04)$. Figure \ref{figure.Homo. example1} represents the vector fields of the solutions of example \ref{example of homo. d and e same signs} before and after the homoclinic saddle connection.
\begin{figure}[ht]
  \centering
  \begin{minipage}[c]{0.5\linewidth}
    \centering
\includegraphics[width=2.3in]{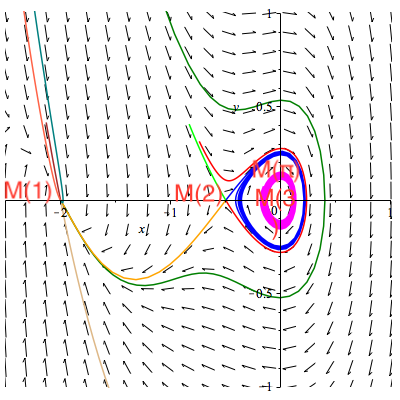}
  \end{minipage}%
\begin{minipage}[c]{0.5\linewidth}
    \centering
\includegraphics[width=2.3in]{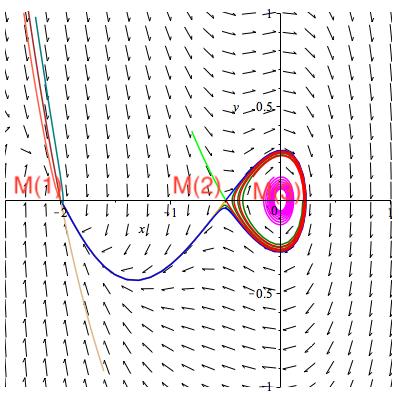}
  \end{minipage}%
  \caption{Before and after a homoclinic bifurcation of $M(2)$: A cycle disappears}
  \label{figure.Homo. example1}
\end{figure}
 In this example, the equilibrium point $E_0 = (0,0)$ is a source and there is a stable limit cycle  by theorem \ref{stability E_0 Heart}, the equilibrium point $E_1 = (- 0.5,0)$ is a saddle by theorem \ref{stability E_1 Heart}, and the equilibrium point $E_2 = (- 2,0)$ is stable by theorem \ref{stability E_2 Heart}. 
 
To prove the existence of certain bifurcations we use the Conley index with $\bbZ_2$ coefficients, connection matrices, and transition  matrices that are  mentioned in the previous section. Also, all $\bbZ_2$ vector spaces are either 0 or of dimension 1. Between two $\bbZ_2$ vector spaces of dimension 1 there are only two linear maps--an isomorphism and zero. 

Recall that the connection matrix  $\Delta$ consists of degree $-1$ maps, and the transition matrix $T$ consists of degree $0$ maps.

The left graph in figure \ref{figure.Homo. example1} represents the flow before the homoclinic bifurcation.  
Let $S^0$ be the compact invariant set in the graph consisting of the equilibria $(-2,0)$, $(-0.5,0)$, and $(0,0)$, a stable limit cycle around the origin, the connection from $(-0.5,0)$ to the limit cycle, the connection from $(-0.5,0)$ to $(-2,0)$, and the disc bounded by the cycle.  Then, by definition \ref{Morse set decomposition}, the  ($<_0$-ordered) Morse decomposition of $S^0$ is a collection $M(S^0) = \{ M(p) \mid p \in \mathcal{P}_0 = \{ 1,2,3, \pi\} \}$ such that $M(1) = \{(-2,0)\}$ which is stable,  $M(2) = \{(-0.5,0)\}$ which is a saddle,  $M(\pi)$ which is the stable limit cycle, and   $M(3) = \{(0,0)\}$ which is a source, with flow ordering  $M(1) <_0 M(2)$, $M(\pi) <_0 M(2)$, $M(\pi) <_0 M(3)$. 

Let us represent $M(1)$ by $1$,  $M(2)$ by $2$, $M(\pi)$ by $\pi$,  $M(3)$ by $3$, and $CH_q(-)$  by $H_q(-)$.
Let $\Delta_0^0$ be the connection matrix before  homoclinic  bifurcation with admissible ordering $<_0$.  Let $H_q^0(k)$ be the Conley index of the Morse set $M(k)$ before the bifurcation such that $k \in  \{1,2,3,\pi\}$ represents the Morse set $M(k)$ and the index $q \in \{0,1,2\}$ is the value where  $H_q(k)$  equals $\bbZ_2$. Referring to table \ref{table basic information} we have $H_q(1)  = \bbZ_2$ $\text{if $q = 0$}$,  $H_q(2) = \bbZ_2$ \text{if $q = 1$},  $H_q (\pi) = \bbZ_2$  \text{if $q = 0,1$}, and $H_q (3) =\bbZ_2$  \text{if $q = 2$}. Denoting a connection orbit between corresponding Morse sets by $\cong$, the  connection matrix $\Delta_0^0$ can be represented as follows:

\begin{equation*}
\Delta_0^0 =
\begin{array}{cc}
 &   \begin{array}{ccccc} H^0_0(1) & H^0_0(\pi)   &  H^0_1(\pi) & H^0_1(2) & H^0_2(3)\end{array}
\\
\begin{array}{c}  H^0_0(1) \\ H^0_0(\pi)  \\ H^0_1(\pi)\\ H^0_1(2)  \\ H^0_2(3)  \end{array} &  
\left(\begin{array}
{c@{\extracolsep{2.9em}}c@{\extracolsep{3.4em}}c@{\extracolsep{3.0em}}c@{\extracolsep{2.7em}}c}
0     &  0   & 0 &  \cong & 0
\\
   0     &  0   & 0 &  \cong & 0
\\
 0     &  0   & 0 &  0 & \cong
 \\
  0     &  0   & 0 &  0 & 0
  \\
  0     &  0   & 0 &  0 & 0
\end{array}  \right)
\end{array}.
\end{equation*}

The right graph in figure \ref{figure.Homo. example1} represents the flow after the homoclinic bifurcation. 
Let $S^1$ be the compact invariant set in the graph consisting of the equilibria $(-2,0)$, $(-0.5,0)$, and $(0,0)$, the two connections from $(-0.5,0)$ to $(-2,0)$, and the connection from the origin to $(-0.5,0)$.  A ($<_1$-ordered) Morse decomposition of $S^1$ is a collection $M(S^1) = \{ M(p) \mid p \in \mathcal{P}_1 = \{ 1,2,3\} \}$ such that, $M(1) = \{(-2,0)\}$ which is stable,  $M(2) = \{(-0.5,0)\}$ which is a saddle, and   $M(3) = \{(0,0)\}$ which is a source, with flow ordering  $M(1) <_1 M(2)$, $M(2) <_1 M(3)$. As there are exactly two connection orbits from $M(2)$ to $M(1)$, $\Delta_1^1(2,1) = 2 \mod(2) = 0$ by corollary \ref{even odd orbits mod 2}, where $\Delta_1^1$ is a connection matrix after homoclinic  bifurcation with admissible ordering $<_1$. Let $H_q^1(k)$ be the Conley index of the Morse set $M(k)$ after bifurcation such that $k \in \{1,2,3\}$ represents the Morse set $M(k)$ and  the index $q \in \{0,1,2\}$ is the value where  $H_q(k)$  equals $\bbZ_2$.  After the bifurcation, the stable limit cycle has disappeared, so we no longer have $M(\pi)$.  The Conley indices of the Morse sets after  bifurcation are still the same as before the bifurcation because there is no change in their stabilities. Then
\begin{equation*}
\Delta_1^1=
\begin{array}{cc}
 &   \begin{array}{ccc} H^1_0(1) &   H^1_1(2)   &   H^1_2(3) \end{array}
\\
\begin{array}{c} H^1_0(1) \\ H^1_1(2)  \\ H^1_2(3) \end{array} &  
\left( 
\begin{array}{c@{\extracolsep{2.9em}}c@{\extracolsep{3.4em}}c}
0     &  0   &  0
\\
   0     &      0    &  \cong 
\\
 0      & 0        &   0
\end{array}  \right)
\end{array}.
\end{equation*}
Let $T_{0,1}$ be the transition matrix for the flow in figure \ref{figure.Homo. example1}. Then
\begin{equation*}
T_{0,1} =
\begin{array}{cc}
 &   \begin{array}{ccccc} H^1_0(1) & H^1_1(2)   &  H^1_2(3) \end{array}
\\
\begin{array}{c} H^0_0(1) \\ H^0_0(\pi)  \\ H^0_1(\pi)\\H^0_1(2)  \\ H^0_2(3)\end{array} &  
\left( \begin{array}{c@{\extracolsep{2.0em}}c@{\extracolsep{2.2em}}c}
 \cong    &  0   & 0 
\\
   *     &  0   & 0 
\\
 0     &  *   & 0 
 \\
  0     &  T(2,2)  & 0
  \\
  0     &  0   &  \cong
\end{array}  \right)
\end{array},
\end{equation*}
where $*$ denotes an undetermined entry. By \eqref{Delta^2 = 0} we have $\Delta_0^0 T_{0,1} +T_{0,1}\Delta_1^1 = 0$, which gives us 
\begin{equation*}
 0 \cdot *\;  + \cong \cdot T(2,2)\; + \cong \cdot\; 0  = 0 \Rightarrow \;\cong \cdot \; T(2,2)  = 0 \Rightarrow T(2,2) = 0.   
\end{equation*}
Then there exists a $\theta_* \in (0.02, 0.04)$ such that there is a homoclinic orbit to $M(2)$ when $\theta = \theta_*$, by corollary 4.4, \cite{Mischaikow88}.

When the signs of the parameters $d$ and $e$ are different,  there are  two prospective heteroclinic saddle connections between the equilibrium points $E_1 = (-d,0)$ and $E_2 = (-e,0)$, and a homoclinic saddle connection at the equilibrium point $E_1 = (- d,0)$ that goes around the origin. For example, consider system \eqref{mvpsystem extended 1st order} with $d = -1 $ and $e = 2$, then we get the following 
\begin{example} \label{example of 2hetero,1homo. d and e different signs}
\begin{equation} \label{Hetro. d = -1, e = 2}
   \begin{aligned} 
    \frac{dx}{dt}& = y,   
    \\
    \frac{dy}{dt}& = - (x^2- \theta) y + \frac{1}{2}x(x-1)(x+2),
\end{aligned} 
\end{equation}
\end{example}
 The equilibrium point $E_1 = (1,0)$ is a saddle by theorem \ref{stability E_1 Heart}, and the equilibrium point $E_2 = (- 2,0)$ is a saddle by theorem \ref{stability E_2 Heart}. Using Maple, we find that as $\theta$ moves from $ \theta_{b_1} = -0.2$ to $\theta_{a_1} = -0.05$, we can expect a first heteroclinic saddle connection between the saddles $E_1 = (1,0)$ and $E_2 =(-2,0)$ with the upper half plane as illustrated in figure \ref{figure.hetero1. example2}. The heteroclinic bifurcation happens at $\theta = \theta_{*1} \in (-0.2 , -0.05)$. Notice that, the equilibrium point $E_0 = (0,0)$ is stable by theorem \ref{stability E_0 Heart} since $\theta < 0$. 
\begin{figure}[ht]
  \centering
  \begin{minipage}[c]{0.5\linewidth}
    \centering
\includegraphics[width=2.3in]{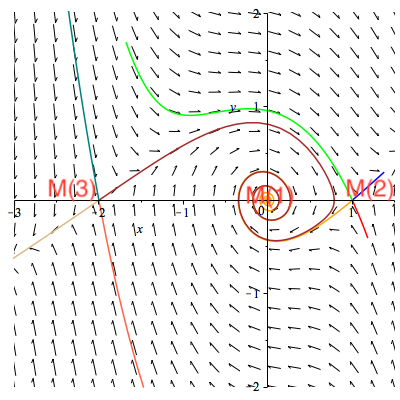}
  \end{minipage}%
\begin{minipage}[c]{0.5\linewidth}
    \centering
\includegraphics[width=2.3in]{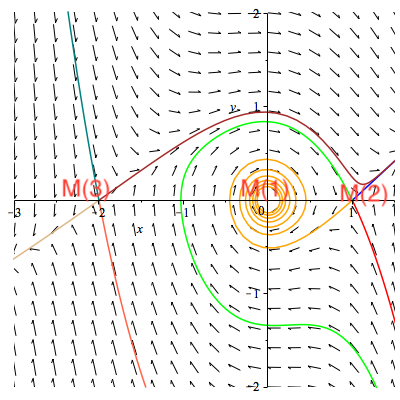}
  \end{minipage}%
  \caption{Before and after the first heteroclinic}
  \label{figure.hetero1. example2}
\end{figure}

The left graph in figure \ref{figure.hetero1. example2} represents the flow before the first heteroclinic bifurcation.
Let $S^2$ be the compact invariant set in the graph consisting of the equilibria $(1,0)$, $(-2,0)$, and $(0,0)$, the connection from $(1,0)$ to the origin, and the connection from $(-2,0)$ to the origin. 
A  ($<_2$-ordered) Morse decomposition of $S^2$ is a collection $M(S^2) = \{ M(p) \mid p \in \mathcal{P}_2 = \{ 1,2,3\} \}$ such that $M(1) = \{(0,0)\}$ which is stable,  $M(2) = \{(1,0)\}$ which is a saddle, and   $M(3) = \{(-2,0)\}$ which is a saddle, with flow ordering  $M(1) <_2 M(2)$ and $M(1) <_2 M(3)$. 
Let $\Delta_2^0$ be the connection matrix before the first heteroclinic  bifurcation with admissible ordering $<_2$.  Then
\begin{equation*}
\Delta_2^0 =
\begin{array}{cc}
 &   \begin{array}{ccc} H^0_0(1) & H^0_1(2)   &  H^0_1(3) \end{array}
\\
\begin{array}{c}  H^0_0(1) \\ H^0_1(2)  \\ H^0_1(3)  \end{array} &  
\left( \begin{array}{c@{\extracolsep{3.4em}}c@{\extracolsep{3.0em}}c}
0     &  \cong  & \cong 
\\
   0     &  0   & 0 
\\
 0     &  0   & 0 
\end{array}  \right).
\end{array}
\end{equation*}

The right graph in  figure \ref{figure.hetero1. example2} represents the flow after the first heteroclinic bifurcation.  
Let $S^3$ be the compact invariant set in the graph consisting of the equilibria $(1,0)$, $(-2,0)$, and the origin, and the connection from $(1,0)$ to the origin.  A ($<_3$-ordered) Morse decomposition of $S^3$ is a collection $M(S^3) = \{ M(p) \mid p \in \mathcal{P}_3 = \{ 1,2,3\} \}$ such that, $M(1) = \{(0,0)\}$ which is stable,  $M(2) = \{(1,0)\}$ which is a saddle, and   $M(3) = \{(-2,0)\}$ which is a saddle, with flow ordering  $M(1) <_3 M(2)$.   

Let $\Delta_3^1$ be the connection matrix after  first heteroclinic  bifurcation with admissible ordering $<_3$. The Conley index of the Morse sets after the bifurcation are still the same as before the bifurcation because there is no change on their stabilities. Then
\begin{equation*}
\Delta_3^1=
\begin{array}{cc}
 &   \begin{array}{ccc} H^1_0(1) &   H^1_1(2)   &   H^1_1(3) \end{array}
\\
\begin{array}{c} H^1_0(1) \\ H^1_1(2)  \\ H^1_1(3) \end{array} &  
\left( \begin{array}{c@{\extracolsep{2.9em}}c@{\extracolsep{3.4em}}c}
0     &  \cong   &  0
\\
   0     &      0    &  0
\\
 0      & 0        &   0
\end{array}  \right)
\end{array}.
\end{equation*}

Let $T_{2,3}$ be the transition matrix for the flow in figure \ref{figure.hetero1. example2}. Then
\begin{equation*}
T_{2,3} =
\begin{array}{cc}
 &   \begin{array}{ccc} H^1_0(1) & H^1_1(2)   &  H^1_1(3) \end{array}
\\
\begin{array}{c} H^0_0(1) \\ H^0_1(2)  \\ H^0_1(3)\end{array} &  
\left( \begin{array}{c@{\extracolsep{2.9em}}c@{\extracolsep{1.9em}}c}
 \cong    &  0   & 0 
\\
   0    &  \cong   & T(2,3)
\\
 0     &  0   &   \cong
\end{array}  \right)
\end{array}.
\end{equation*}
By \eqref{Delta^2 = 0} we have $\Delta^0_2 T_{2,3} +T_{2,3} \Delta^1_3 = 0$ which gives us 
\begin{equation*}
  \cong \cdot\; T(2,3)\; + \cong \cdot \cong \; = 0 \Rightarrow T(2,3) = \;\cong 
\end{equation*}
Then, by theorem 3.13, \cite{Reineck88}  we have  $\partial \bigl(M(3), M(2)\bigr) \neq 0$, which implies to $C\bigl(M(3), M(2)\bigr) \neq \emptyset$. Therefore, the saddles are connected for some parameter value $\theta_{*1} \in (-0.2, -0.05)$.

The second bifurcation in this example is a homoclinic saddle connection at the equilibrium point $E_1 = (1,0)$ when $\theta$ moves from $ \theta_{b_2} = 0.1$ to $\theta_{a_2} = 0.2$,  as  is illustrated in figure \ref{figure.homo1. example2}. The homoclinic bifurcation happens at $\theta = \theta_{*2} \in (0.1 , 0.2)$. In this case, the equilibrium point $E_0 = (0,0)$ is unstable by theorem \ref{stability E_0 Heart} since $\theta > 0$. The  homoclinic saddle connection happened after breaking the stable limit cycle which we proved the existence of in theorem \ref{stability E_0 Heart}. Also, there is a bifurcation  as $\theta$ passes from negative to positive where the equilibrium point $E_0$ is changed from unstable to stable.  
%
\begin{figure}[ht]
  \centering
  \begin{minipage}[c]{0.5\linewidth}
    \centering
\includegraphics[width=2.3in]{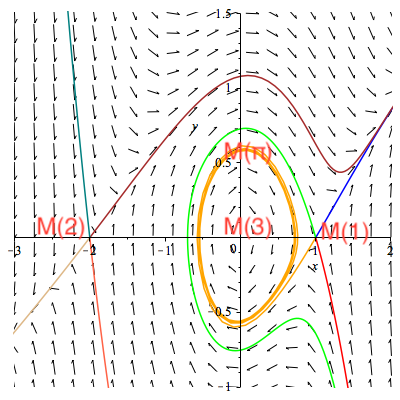}
  \end{minipage}%
\begin{minipage}[c]{0.5\linewidth}
    \centering
\includegraphics[width=2.3in]{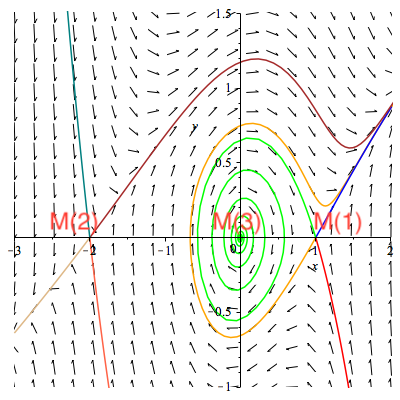}
  \end{minipage}%
  \caption{Before and after a homoclinic bifurcation of $M(1)$: A cycle disappears}
  \label{figure.homo1. example2}
\end{figure}

The left graph in the figure \ref{figure.homo1. example2} represents the flow before the first heteroclinic bifurcation.
Let $S^4$ be the compact invariant set in the graph consisting of the equilibria $(-2,0)$, $(1,0)$, and $(0,0)$, a stable limit cycle around the origin, the connection from $(1,0)$ to the limit cycle, and the disc bounded by the cycle. A ($<_4$-ordered) Morse decomposition of $S^4$ is a collection $M(S^4) = \{ M(p) \mid p \in \mathcal{P}_4 = \{ 1,2,3,\pi \} \}$ such that $M(1) = \{(1,0)\}$ is a saddle, $M(2) = \{(-2,0)\}$  is a saddle,  $M(3) = \{(0,0)\}$  is a source, and $M(\pi)$ is stable limit cycle, with flow ordering  $M(\pi) <_4 M(1)$ and $M(\pi) <_4 M(3)$.

Let $\Delta_4^0$ be the connection matrix before homoclinic  bifurcation with admissible ordering $<_4$, then

\begin{equation*}
\Delta_4^0 =
\begin{array}{cc}
 &   \begin{array}{ccccc}   H^0_0(\pi)   &  H^0_1(\pi) & H^0_1(1) & H^0_1(2) & H^0_2(3)\end{array}
\\
\begin{array}{c}   H^0_0(\pi)  \\ H^0_1(\pi)\\  H^0_1(1)\\H^0_1(2)  \\ H^0_2(3)  \end{array} &  
\left( \begin{array}{c@{\extracolsep{2.9em}}c@{\extracolsep{3.4em}}c@{\extracolsep{3.0em}}c@{\extracolsep{2.7em}}c}
0     &  0   & \cong &  0 & 0
\\
   0     &  0   & 0 & 0 & \cong
\\
 0     &  0   & 0 &  0 & 0
 \\
  0     &  0   & 0 &  0 & 0
  \\
  0     &  0   & 0 &  0 & 0
\end{array}  \right)
\end{array}.
\end{equation*}
The right graph in  figure \ref{figure.homo1. example2} represents the flow after the first heteroclinic bifurcation. 
Let $S^5$ be the compact invariant set in the graph consisting of the equilibria $(-2,0)$, $(1,0)$, and $(0,0)$, as well as the connections from the origin to $(1,0)$. 
A ($<_5$-ordered) Morse decomposition of $S^5$ is a collection $M(S^5) = \{ M(p) \mid p \in \mathcal{P}_5 = \{ 1,2,3 \} \}$ such that $M(1) = \{(1,0)\}$ which is a saddle, $M(2) = \{(-2,0)\}$ which is a saddle, and  $M(3) = \{(0,0)\}$ which is a source, with flow ordering  $M(1) <_5 M(3)$.

Let $\Delta_5^1$ be the connection matrix after homoclinic  bifurcation with admissible ordering $<_5$, then
 \begin{equation*}
\Delta_5^1=
\begin{array}{cc}
 &   \begin{array}{ccc} H^1_1(1) &   H^1_1(2)   &   H^1_2(3) \end{array}
\\
\begin{array}{c} H^1_1(1) \\ H^1_1(2)  \\ H^1_2(3) \end{array} &  
\left( \begin{array}{c@{\extracolsep{2.9em}}c@{\extracolsep{3.4em}}c}
0     &  0   &  \cong
\\
   0     &      0    &  0
\\
 0      & 0        &   0
\end{array}  \right)
\end{array}.
\end{equation*}
Let $T_{4,5}$ be the transition matrix for the flow in figure \ref{figure.homo1. example2}. Then
\begin{equation*}
T_{4,5} =
\begin{array}{cc}
 &   \begin{array}{ccccc} H^1_1(1) & H^1_1(2)   &  H^1_2(3) \end{array}
\\
\begin{array}{c} H^0_0(\pi)  \\ H^0_1(\pi)\\ H^0_1(1) \\ H^0_1(2)  \\ H^0_2(3)\end{array} &  
\left( \begin{array}{c@{\extracolsep{2.0em}}c@{\extracolsep{3.0em}}c}
0   &  0   & 0 
\\
   *     &  *   & 0 
\\
  T(1,1)    &  *   & 0 
 \\
  *     &   \cong   & 0
  \\
  0     &  0   &  \cong
\end{array}  \right)
\end{array}.
\end{equation*}
By \eqref{Delta^2 = 0} we have $\Delta^0_4 T_{4,5} +T_{4,5} \Delta^1_5 = 0$ which give us
\begin{equation*}
    \cong \cdot\;  T(1,1)  = 0 \Rightarrow T(1,1) = 0.
\end{equation*}
 
 Then there exists $\theta_{*2} \in (0.1 , 0.2)$ such that  there exists a homoclinic orbit to $M(1)$ when $\theta = \theta_{*2}$, by corollary 4.4 \cite{Mischaikow88}.

Another bifurcation is the second heteroclinic saddle connection  when $\theta$ moves from $ \theta_{b_3} = 1.1$ to $\theta_{a_3} = 1.2$ between the saddles $E_1 = (1,0)$ and $E_2 =(-2,0)$  with the lower half plane as illustrated in figure \ref{figure.hetero2. example2}. The heteroclinic bifurcation happens at $\theta = \theta_{*3} \in (1.1 , 1.2)$.
The equilibrium point $E_0 = (0,0)$ is unstable by theorem \ref{stability E_0 Heart} as $\theta > 0$.
\begin{figure}[ht]
  \centering
  \begin{minipage}[c]{0.5\linewidth}
    \centering
\includegraphics[width=2.3in]{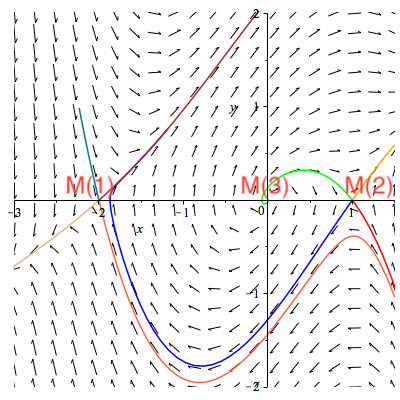}
  \end{minipage}%
\begin{minipage}[c]{0.5\linewidth}z
    \centering
\includegraphics[width=2.3in]{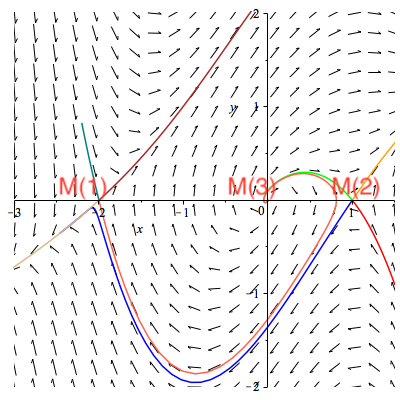}
  \end{minipage}%
  \caption{Before and after the second heteroclinic }
  \label{figure.hetero2. example2}
\end{figure}

The left graph in the figure \ref{figure.hetero2. example2} represents the flow before the first heteroclinic bifurcation.
Let $S^6$ be the compact invariant set in the graph consisting of the equilibria $(1,0)$, $(-2,0)$, and  $(0,0)$, as well as the connection from $(1,0)$ to the origin. A  ($<_6$-ordered) Morse decomposition of $S^6$ is a collection $M(S^6) = \{ M(p) \mid p \in \mathcal{P}_6 = \{ 1,2,3\} \}$ such that $M(1) = \{(-2,0)\}$ which is a saddle,  $M(2) = \{(1,0)\}$ which is a saddle, and   $M(3) = \{(0,0)\}$ which is a source, with flow ordering  $M(2) <_6 M(3)$. 

Let $\Delta_6^0$ be the connection matrix before  the second heteroclinic  bifurcation with admissible ordering $<_6$. Then
\begin{equation*}
\Delta_6^0 =
\begin{array}{cc}
 &   \begin{array}{ccc} H^0_1(1) & H^0_1(2)   &  H^0_2(3) \end{array}
\\
\begin{array}{c}  H^0_1(1) \\ H^0_1(2)  \\ H^0_2(3)  \end{array} &  
\left( \begin{array}{c@{\extracolsep{2.9em}}c@{\extracolsep{3.4em}}c}
0     & 0  & 0 
\\
  0 &  0   &  \cong 
\\
 0     &  0   & 0 
\end{array}  \right)
\end{array}.
\end{equation*}
 
The right graph in figure \ref{figure.hetero2. example2} represents the flow after the first heteroclinic bifurcation.
Let $S^7$ be the compact invariant set in the graph consisting of the equilibria $(1,0)$, $(-2,0)$, and the origin, the connection from $(1,0)$ to the origin, and the connection from $(-2,0)$ to the origin. A  ($<_7$-ordered) Morse decomposition of $S^7$ is a collection $M(S^7) = \{ M(p) \mid p \in \mathcal{P}_7 = \{ 1,2,3\} \}$ such that $M(1) = \{(-2,0)\}$ is a saddle,  $M(2) = \{(1,0)\}$ is a saddle, and   $M(3) = \{(0,0)\}$ is a source, with flow ordering  $M(1) <_7 M(3)$ and $M(2) <_7 M(3)$.
  
 Let $\Delta_7^1$ be the connection matrix after  second heteroclinic  bifurcation with admissible ordering $<_7$, then
 \begin{equation*}
\Delta_7^1=
\begin{array}{cc}
 &   \begin{array}{ccc} H^1_1(1) &   H^1_1(2)   &   H^1_2(3) \end{array}
\\
\begin{array}{c} H^1_1(1) \\ H^1_1(2)  \\ H^1_2(3) \end{array} &  
\left( \begin{array}{c@{\extracolsep{2.9em}}c@{\extracolsep{3.4em}}c}
0     &  0   &   \cong
\\
   0     &      0    &   \cong
\\
 0      & 0        &   0
\end{array}  \right)
\end{array}.
\end{equation*}
Let $T_{6,7}$ be the transition matrix for the flow in figure \ref{figure.hetero2. example2}. Then
 \begin{equation*}
T_{6,7} =
\begin{array}{cc}
 &   \begin{array}{ccc} H^1_1(1) & H^1_1(2)   &  H^1_2(3) \end{array}
\\
\begin{array}{c} H^0_1(1) \\ H^0_1(2)  \\ H^0_2(3)\end{array} &  
\left( \begin{array}{c@{\extracolsep{2.4em}}c@{\extracolsep{2.0em}}c}
 \cong    &  T(1,2)   & 0 
\\
   0    &  \cong   & 0 
\\
 0     &  0   &   \cong
\end{array}  \right)
\end{array}.
\end{equation*}
By \eqref{Delta^2 = 0} we have  $\Delta^0_6 T_{6,7} +T_{6,7} \Delta^1_7 = 0$ which gives us 
\begin{equation*}
  \cong \cdot\; T(1,2)\; + \cong \cdot \cong \; = 0, \Rightarrow T(1,2) = \;\cong.  
\end{equation*}
Then by theorem 3.13  \cite{Reineck88} we have  $\partial \bigl(M(2),M(1)\bigr) \neq ~0$, which implies to $C\bigl(M(2),M(1),\bigr) \neq \emptyset$. Therefore, the saddles are connected for some parameter value $\theta_{*3} \in (1.1 , 1.2)$.

\bibliographystyle{unsrt}  


\end{document}